\documentclass{amsart}

\usepackage{amssymb}
\usepackage{amsmath}
\usepackage{verbatim}
\usepackage{hyperref}

\newtheorem{theorem}{Theorem}[section]
\newtheorem{lemma}[theorem]{Lemma}

\numberwithin{equation}{section}

\newcommand{\C}{\ensuremath{\mathbb{C}^n}}

 \newcommand{\Z}{\ensuremath{\mathbb{Z}^{2n}}}

\newcommand{\Ftwophi}{\ensuremath{F_\phi ^2 }}

\newcommand{\incn}{\ensuremath{\int_{\C}}}

\begin{document}
\title[Schatten class Toeplitz operators]{Schatten class Toeplitz operators on generalized Fock spaces}

\author{Joshua Isralowitz, Jani Virtanen, \and Lauren Wolf}

\begin{abstract}
In this paper we characterize the Schatten $p$ class membership of Toeplitz operators with positive measure symbols acting on generalized Fock spaces for the full range $0 < p < \infty$.
\end{abstract}



\address{Department of Mathematics, University at Albany, Albany, NY 12222, USA}
\email{jisralowitz@albany.edu}

\address{Department of Mathematics, University of Reading, Reading RG6 6AX, UK}
\email{j.a.virtanen@reading.ac.uk}

\address{Department of Mathematics, University at Albany, Albany, NY 12222, USA}
\email{lwolf-christensen@albany.edu}

\maketitle

\section{Introduction}
 Let $d^c =  \frac{i}{4} (\overline{\partial} - \partial)$ and let $d$ be the usual exterior derivative. Throughout the paper, let $\phi \in C^2(\C)$ be a real valued function on $\C$  such that \begin{equation} c \omega_0 < d d^c \phi < C \omega_0 \label{PhiCond} \end{equation} holds uniformly pointwise on $\C$ for some positive constants $c$ and $C$ (in the sense of positive $(1, 1)$ forms) where $\omega_0 = d d^c |\cdot |^2$ is the standard Euclidean K\"{a}hler form.

 Define $F^2_\phi$ to be the set of entire functions such that
$$
	\int_{\C} |f(z)|^2 e^{-2\phi(z)} dv(z) < \infty.
$$
Denote by $P$ the orthogonal projection of $L^2(e^{-2\phi}dv)$ onto $F^2_\phi$. For a positive measure $\mu$, define the Toeplitz operator $T_\mu : F^2_\phi\to F^2_\phi$ with symbol $\mu$ by setting
$$
	T_\mu f(z) = \int_{\mathbb{C}^n} K(z, w) f(w) e^{-2\phi(w)} d\mu(w),
$$
where $K$ stands for the reproducing (or Bergman) kernel of $F^2_\phi$, that is,
$$
	K(z, w) = \sum_{k=1}^\infty f_k(z) \overline{f_k(w)},
$$
where $\{f_k\}$ is any orthonormal basis for $F^2_\phi$. In the next section we list some recent results on the reproducing kernel (see~\cite{SV}), which will be crucial to the proofs of our main results on Schatten class properties of Toeplitz operators.

In \cite{IZ} (see also a recent monograph of Zhu \cite{Z}), Toeplitz operators were considered in the setting of the standard weighted Fock spaces, that is, when $\phi(z) = \frac{\alpha}{2} |z|^2$ for $\alpha > 0$, and characterizations of bounded, compact and Schatten class Toeplitz operators with positive measure symbols were provided (moreover, see \cite{SV} for a similar characterization of bounded and compact Toeplitz operators with positive measure symbols on $\Ftwophi$). In particular, the Schatten class membership of these Toeplitz operators was characterized in terms of the heat (Berezin) transform of the symbol and in terms of the averaging function $\mu(B(\cdot, r))$.

In this paper we will provide very similar characterizations of the Schatten class membership of these Toeplitz operators.  Note that unlike the classical Fock space setting where one can utilize explicit formulas for the reproducing kernel, we instead must rely on some known estimates on the behavior of the reproducing kernel (see the first three lemmas in the next section). The proofs of our characterizations will (as usual) be divided into the two cases  $0< p \leq 1$ (which will be dealt with in Section \ref{case1})  and $p>1$ (which will be dealt with in Section \ref{case2}.)

Let us note that one can easily write the so called ``Fock-Sobolev spaces" from \cite{CZ} as a weighted Fock space $\Ftwophi$ with $\phi$ satisfying (\ref{PhiCond}), so that in particular our results immediately apply to these spaces (see \cite{I} for more details.) 

Finally, note that we will often use the notation $A \lesssim B$ for two nonnegative quantities $A$ and $B$ if $A \leq C B$ for an unimportant constant $C$.  Moreover, the notation $A \gtrsim B$ and $A \approx B$ will have similar meanings.

\section{The case $0 < p  \leq 1$}\label{case1}  In this section we will characterize Schatten $p$ class $T_\mu$ for the case $0 < p \leq 1$.    We will often use the following three lemmas from \cite{SV}.

\begin{lemma} \label{OffDiagDecayLem} If $K$ is the reproducing kernel of $\Ftwophi$ then there exists $\epsilon_0 > 0$ where \begin{equation*} e^{-\phi(w)} |K(z, w)| e^{-\phi(z)} \lesssim e^{-\epsilon_0 |z - w|} \end{equation*} \end{lemma}

\begin{lemma} \label{LowerBoundForKernel} There exists $\delta > 0$ where \begin{equation*}  e^{-\phi(w)} |K(z, w)| e^{-\phi(z)} \gtrsim 1 \end{equation*} for all $w \in B(z, \delta)$.   In particular, $K(z, z) e^{-2\phi(z)} \approx 1$.  \end{lemma}

\begin{lemma}\label{PointwiseBound} If $r > 0$ then there exists $C_r > 0$ independent of $f \in \Ftwophi$ where \begin{equation*} |f(z) e^{- \phi(z)}|^2    \lesssim C_r \int_{B(z, r)} |f(w) e^{-  \phi(w)}|^2 \, dv(w). \end{equation*} \end{lemma}

The basic outline of our arguments will be similar to the proofs in \cite{Z} for the classical Fock space (which themselves are based on the ideas in the seminal paper \cite{L}). However, note that in the classical Fock space situation (when $\phi(z) = \frac{\alpha}{2} |z|^2$ for some $\alpha > 0$), we have that \begin{equation*}  e^{-\phi(w)} |K(z, w)| e^{-\phi(z)} = e^{- \frac{\alpha}{2} |z - w|^2}. \end{equation*} Because of this, we will often have to make modifications to the arguments in \cite{Z}.

Now assume $\mu$ satisfies the condition that \begin{equation} \incn e^{-\gamma |z - w|} \, d\mu(w) < \infty \label{propertyM} \end{equation} for all $\gamma > 0$ and $z \in \C$. Note that Lemma \ref{OffDiagDecayLem} immediately tells us that $T_\mu$ is well defined on the span of $\{K(\cdot, w) : w \in \C\}$ if $\mu$ satisfies condition (\ref{propertyM}), so in particular $T_\mu$ is densely defined.

  Let $\widetilde{\mu}$ be the Berezin transform of $\mu$ defined by $\widetilde{\mu}(z) := \langle T_\mu k_z, k_z \rangle_{\Ftwophi}$  where $k_z$ is the normalized reproducing kernel of $\Ftwophi$. Note that (as one would expect), (\ref{propertyM}) in conjunction with Fubini's theorem gives us that \begin{equation*} \widetilde{\mu}(z) = \incn |k_z (w)|^2 e^{-2\phi(w)} \, d\mu(w). \end{equation*}

If $r > 0$ then for the remainder of this paper we will let $\{a_m\}$ denote any fixed arrangement of the lattice $r \Z$ (which is canonically treated as a subset of $\C$.)

\begin{lemma} \label{EquivOfBerTypeFunc} Suppose that $\mu \geq 0, r > 0, $ and $0 < p \leq 1.$  If $\mu$ satisfies condition $(\ref{propertyM})$, then the following are equivalent: \begin{itemize}
\item[(a)] $\widetilde{\mu} \in L^p(\C, dv)$
\item[(b)] $\mu(B(\cdot, r)) \in L^p(\C, dv)$
\item[(c)] $\{\mu(B(a_m, r))\} \in l^p $
\end{itemize}
\end{lemma}

\begin{proof}

The equivalence of $(b)$ and $(c)$ for any $r > 0$ was proved in \cite{Z}, where it was also proved that $(b)$ and $(c)$ are in fact independent of $r > 0$.  Thus, we will complete the proof by showing that $(a) \Longleftrightarrow (c)$ for some $r > 0$.

First assume that $(c)$ is true.  Then by Lemma \ref{PointwiseBound} we have that \begin{align*} \widetilde{\mu}(z) & = \incn |k_z (w)|^2 e^{-2\phi(w)} \, d\mu(w) \\ & \lesssim \incn \left(\int_{B(w, r)} |k_z(u)|^2 e^{- 2 \phi(u)} \, dv(u) \right)\, d\mu(w) \\ & = \incn \left(\incn \chi_{B(u, r)} (w) |k_z(u)|^2 e^{- 2 \phi(u)}\, dv(u) \right)\, d\mu(w) \\  & = \incn \mu(B(u, r)) |k_z(u)|^2 \, e^{-2\phi(u)} dv(u) \\ & \lesssim \sum_{m = 1}^\infty \int_{B(a_m, r)} \mu(B(u, r)) e^{-\epsilon_0 |z - u|} \, dv(u) \end{align*} where the last inequality and $\epsilon_0 > 0$ follow from Lemma \ref{OffDiagDecayLem}. However, $B(u, r) \subset B(a_m, 2r)$ if $u \in B(a_m, r)$ so that \begin{equation} \widetilde{\mu}(z) \lesssim \sum_{m = 1}^\infty \mu(B(a_m, 2r)) \int_{B(a_m, r)}  e^{-\epsilon_0 |z - u|} \, dv(u). \label{EquivOfBerTypeFuncEq1} \end{equation}

Furthermore, since $0 < p \leq 1$, equation $(\ref{EquivOfBerTypeFuncEq1})$ gives us that \begin{equation} \incn (\widetilde{\mu}(z))^p \, dv(z) \lesssim \sum_{m = 1}^\infty (\mu(B(a_m, 2r)))^p \incn \left( \int_{B(a_m, r)}  e^{-\epsilon_0 |z - u|} \, dv(u) \right)^p \, dv(z) \label{EquivOfBerTypeFuncEq2}. \end{equation}

However, we can easily estimate the right hand side of $(\ref{EquivOfBerTypeFuncEq2})$ as follows.  First, it is obvious that \begin{equation}   \int_{B(a_m, 2r)} \left( \int_{B(a_m, r)}  e^{-\epsilon_0 |z - u|} \, dv(u) \right)^p \, dv(z) \lesssim r^{2n(p + 1)} \label{EquivOfBerTypeFuncEq3}.\end{equation} On the other hand, if $|z - a_m| \geq 2r$ and $|u - a_m| \leq r$, then \begin{equation*}  |z - a_m| \leq |z - u| + |u - a_m| \leq |z - u| + r \end{equation*} so that \begin{equation*} |z - u| \geq |z - a_m| - r \geq \frac{1}{2} |z - a_m|.  \end{equation*}  Thus, we have that
\begin{multline}
	\int_{\C \backslash B(a_m, 2r)} \left( \int_{B(a_m, r)}  e^{-\epsilon_0 |z - u|} \, dv(u) \right)^p \, dv(z) \\
	\lesssim r^{2np} \int_{\C \backslash B(a_m, 	2r)} e^{- \frac{\epsilon_0 p}{2} |z - a_m|} \, dv(z)
	\lesssim r^{2np} \label{EquivOfBerTypeFuncEq4}.
\end{multline}

Finally, combining $(\ref{EquivOfBerTypeFuncEq2})$ with $(\ref{EquivOfBerTypeFuncEq3})$ and $(\ref{EquivOfBerTypeFuncEq4})$ we have that \begin{equation*} \incn (\widetilde{\mu}(z))^p \, dv(z) \leq C_r  \sum_{m = 1}^\infty (\mu(B(a_m, 2r)))^p < \infty \end{equation*} for some $C_r > 0$ since $(c)$ is independent of $r > 0$.

We now complete the proof by showing that $(a) \Rightarrow (c)$ for $r  = \frac{\delta}{2}$ where $\delta$ is from Lemma \ref{OffDiagDecayLem}.  In particular, \begin{equation*} \incn (\widetilde{\mu}(z))^p \, dv(z) \gtrsim \sum_{m = 1}^\infty \int_{B(a_m, \delta/2)} (\widetilde{\mu}(z))^p \, dv(z). \end{equation*} Moreover, if $z \in B(a_m, \delta/2)$ then Lemma \ref{LowerBoundForKernel} gives us that \begin{equation*} \widetilde{\mu}(z) \geq \int_{B(a_m, \delta/2) } |k_z(u)|^2 \, e^{-2\phi(u)} \, d\mu(u) \gtrsim \mu(B(a_m, \delta/2)) \end{equation*} which immediately implies that $(c)$ is true with $r = \frac{\delta}{2}.$ \end{proof}

\begin{lemma} \label{BerInLpImpSchat}  Suppose that $\mu \geq 0$ and $\mu$ satisfies condition $(\ref{propertyM})$.  Then
\begin{itemize}
\item [(a)]  $ \ T_\mu \in S_p$ if $\widetilde{\mu} \in L^p(\C, dv)$ and $0 < p \leq 1,$
\item [(b)] $ \ \widetilde{\mu} \in L^p(\C, dv)$ if $T_\mu \in S_p$ and $1 \leq p < \infty$.
\end{itemize}
\end{lemma}

\begin{proof} Since $\widetilde{\mu} \in L^p(\C, dv)$ implies that $\{\mu(B(a_m, r))\}$ is bounded by Lemma \ref{EquivOfBerTypeFunc}, we first of all have that $T_\mu$ is bounded on $\Ftwophi$ by Theorem $1$ in \cite{SV}.  Furthermore, since $\sqrt{K(z, z)} \approx e^{\phi(z)}$, one can repeat virtually word for word the arguments on pp.~96--97 in \cite{Z} to complete the proof.\end{proof}

We will need one more lemma before we prove the main result of this section.  Note that the proof of this lemma follows from standard ideas in frame theory (see the proof of Lemma $3.7$ in \cite{LR} for example), though we will include its simple proof for the sake of completion.

\begin{lemma} \label{FrameOpBdd}  Let $r > 0$ and let $\{e_m\}$ be any orthonormal basis for $\Ftwophi$.  If $\{\xi_m\} \subset r \Z$ and $A$ is the operator on $\Ftwophi$ defined by $A e_m := k_{\xi_m}$ then $A$ extends to a bounded operator on all of $\Ftwophi$ whose operator norm is bounded above by a constant that only depends on $r$.
\end{lemma}

\begin{proof}  If $f, g \in \Ftwophi$, then the Cauchy-Schwarz inequality, Lemma \ref{PointwiseBound}, and the reproducing property  gives us that \begin{align*} |\langle Af, g\rangle| & \leq \sum_{m = 1} ^\infty |\langle f, e_m\rangle_\phi   \langle k_{\xi_m}, g\rangle_\phi| \\ & \leq \|f\|_{\Ftwophi} \left(\sum_{m = 1} ^\infty  |\langle k_{\xi_m}, g\rangle_\phi|^2\right)^\frac{1}{2} \\ & \approx  \|f\|_{\Ftwophi} \left(\sum_{m = 1} ^\infty |g(\xi_m) e^{-\phi(\xi_m)}|^2 \right)^\frac{1}{2} \\ & \lesssim \|f\|_{\Ftwophi} \left(\sum_{m = 1} ^\infty \int_{B(\xi_m, r)} |g(u) e^{-\phi(u)}|^2 \, dv(u)\right)^\frac{1}{2} \lesssim \|f\|_{\Ftwophi} \|g\|_{\Ftwophi} \end{align*} \end{proof}

Note that $\|\cdot \|_{S_p}$ is not a norm when $p < 1$.  However, it is well known that if $A$ and $B$ are compact, then \begin{equation*} s_{m + n - 1} (A + B) \leq s_m(A) + s_n(B) \end{equation*} where $s_k(T)$ is the $k^\text{th}$ singular value of a compact operator $T$.  Thus, it is easy to see that for all $0 < p \leq 1$ we have \begin{equation} \label{QuasiTriIneq} \|A + B \|_{S_p} ^p \leq 2 (\|A\|_{S_p}^p + \|B\|_{S_p} ^p)  \end{equation} for any $A, B \in S_p$.

\begin{theorem} \label{MainThmForpleq1}  Suppose $\mu \geq 0, 0 < p \leq 1, $ and $\mu$ satisfies condition $(\ref{propertyM})$. Then the following are equivalent:
\begin{itemize}
\item[(a)] $T_\mu \in S_p$
\item[(b)] $\widetilde{\mu} \in L^p(\C, dv)$
\item[(c)] $\mu(B(\cdot, r)) \in L^p(\C, dv)$
\item[(d)] $\{\mu(B(a_m, r))\} \in l^p $
\end{itemize}
\end{theorem}

\begin{proof} By Lemmas \ref{EquivOfBerTypeFunc} and \ref{BerInLpImpSchat}, it is enough to show that $(a) \Rightarrow (d)$ for $r = \delta$ where again $\delta$ is from Lemma \ref{LowerBoundForKernel}. For that matter, pick some large $R > 2\delta$ (to be determined later) and partition $\{a_m\}$ into $N$ sublattices $\{\xi_m\}$ where $m \neq k \Rightarrow |\xi_m - \xi_k| > R$.  Furthermore, let \begin{equation*} \nu := \sum_{m = 1}^\infty \mu \chi_{B(\xi_m, \delta)}. \end{equation*}  Clearly $T_\nu \leq T_\mu$ so that $\|T_\nu\|_{S_p} \leq \|T_\mu\|_{S_p}$.

Now fix any orthonormal basis $\{e_m\}$ of $\Ftwophi$ and let $A$ be the operator on $\Ftwophi$ defined by $A e_m := k_{\xi_m}$ (which by Lemma \ref{FrameOpBdd} has operator norm that is bounded above by a constant that is independent of $\{\xi_m\}$).  Now let $T := A^* T_\nu A$ so that \begin{equation*} \|T\|_{S_p}  \lesssim \|T_\nu\|_{S_p} \leq \|T_\mu\|_{S_p}. \end{equation*}  Furthermore, define \begin{equation*} Df := \sum_{m = 1}^\infty \langle Te_m, e_m\rangle_\phi \langle f, e_m\rangle_\phi e_m  \end{equation*} and $E := T - D$ so that by $(\ref{QuasiTriIneq})$ we have \begin{equation*} \|T_\mu\|_{S_p} ^p \gtrsim \|T\|_{S_p} ^p \geq  \frac{1}{2} \|D\|_{S_p}^p - \|E\|_{S_p}^p. \end{equation*}   Then since $D$ is diagonal, we have from Lemma \ref{LowerBoundForKernel} that \begin{align} \|Df \|_{S_p} ^p & = \sum_{m = 1}^\infty \langle Te_m, e_m \rangle_\phi ^p \label{MainThmForpleq1DiagEst} \\     & =    \sum_{m = 1}^\infty \langle T_\nu k_{\xi_m}, k_{\xi_m} \rangle_\phi  ^p \nonumber \\ & = \sum_{m = 1}^\infty \left(\incn |k_{\xi_m} (u)|^2 \, e^{-2\phi (u)} d\nu(u)\right)^p \nonumber \\ & \geq  \sum_{m = 1}^\infty \left(\int_{B(\xi_m, \delta)} |k_{\xi_m} (u)|^2 \, e^{-2\phi (u)} d\mu(u)\right)^p \nonumber \\ & \geq C_1 \sum_{m = 1}^\infty \mu(B(\xi_m, \delta))^p \nonumber \end{align} for some $C_1 > 0$ independent of $N$.

We will now get an upper bound for $\|E\|_{S_p} ^p$.   By Proposition $1.29$ in \cite{Z2} and Lemma \ref{OffDiagDecayLem}, we have that \begin{align} \|E\|_{S_p}^p &\leq \sum_{m \neq k}|\langle T e_m, e_k\rangle_\phi|^p    \nonumber \\ & = \sum_{m \neq k} |\langle T_\nu k_{\xi_m}, k_{\xi_k} \rangle_\phi|^p  \nonumber \\ & \leq \sum_{m \neq k} \left(\incn |k_{\xi_m}(u) k_{\xi_k} (u)| e^{-2\phi(u)} \, d\nu(u) \right)^p \nonumber \\ & \lesssim \sum_{m \neq k} \left(\incn e^{- \epsilon_0 |u - \xi_m|} e^{-\epsilon_0 |u - \xi_k|} \, d\nu(u) \right)^p \label{MainThmForpleq1OffDiagEst1}.\end{align}

Now if $m \neq k$ then $|\xi_m - \xi_k| \geq R$.  Thus, if $|u - \xi_m| \leq \frac{R}{2}$, then the triangle inequality gives us that $|u - \xi_k| \geq \frac{R}{2}$.  Plugging this into (\ref{MainThmForpleq1OffDiagEst1}) gives us that \begin{equation} \|E\|_{S_p}^p \lesssim e^{- \frac{\epsilon_0 p R}{2}} \sum_{m \neq k} \left(\incn e^{- \frac{\epsilon_0}{2} |u - \xi_m|} e^{-\frac{\epsilon_0}{2} |u - \xi_k|} \, d\nu(u) \right)^p. \label{MainThmForpleq1OffDiagEst2} \end{equation}

Since $\nu$ is supported on $\bigcup_{j = 1}^\infty B(\xi_j, \delta)$, we have that \begin{equation}\incn e^{- \frac{\epsilon_0}{2} |u - \xi_m|} e^{-\frac{\epsilon_0}{2} |u - \xi_k|} \, d\nu(u) = \sum_{j = 1}^\infty \int_{B(\xi_j, \delta)} e^{- \frac{\epsilon_0}{2} |u - \xi_m|} e^{-\frac{\epsilon_0}{2} |u - \xi_k|} \, d\mu(u). \label{MainThmForpleq1OffDiagEst3} \end{equation} Moreover, if $j \neq m$ and $|u - \xi_j| < \delta$ then \begin{equation*} |\xi_j - \xi_m| \leq |\xi_j - u| + |u - \xi_m| \leq \delta + |u - \xi_m|.  \end{equation*} Thus, as $|\xi_j - \xi_m| > R \geq 2\delta$ we have that \begin{equation*}|u - \xi_m| \geq  |\xi_j - \xi_m| - \delta \geq \frac{1}{2} |\xi_j - \xi_m| \end{equation*} and clearly we have a similar estimate for $|u - \xi_k|$.

Plugging this into $(\ref{MainThmForpleq1OffDiagEst3})$ gives us that \begin{equation} \incn e^{- \frac{\epsilon_0}{2} |u - \xi_m|} e^{-\frac{\epsilon_0}{2} |u - \xi_k|} \, d\nu(u) \leq  \sum_{j = 1}^\infty \mu(B(\xi_j, \delta)) e^{-\frac{\epsilon_0  }{4} |\xi_j - \xi_m|} e^{-\frac{\epsilon_0  }{4} |\xi_j - \xi_k|}. \label{MainThmForpleq1OffDiagEst4} \end{equation} Thus, since $0 < p \leq 1$, we can plug $(\ref{MainThmForpleq1OffDiagEst4})$ into $(\ref{MainThmForpleq1OffDiagEst2})$ to get that \begin{align*} \|E\|_{S_p} ^p & \lesssim e^{- \frac{\epsilon_0 p R}{2}}\sum_{j = 1}^\infty \mu(B(\xi_j, \delta))^p  \sum_{m \neq k} e^{-\frac{\epsilon_0 p }{4} |\xi_j - \xi_m|} e^{-\frac{\epsilon_0 p }{4} |\xi_j - \xi_k|} \\ & \leq e^{- \frac{\epsilon_0 p R}{2}}\sum_{j = 1}^\infty \mu(B(\xi_j, \delta))^p  \left(\sum_{m = 1}^\infty e^{-\frac{\epsilon_0 p }{4} |\xi_j - \xi_m|} \right)^2 \\ & \lesssim e^{- \frac{\epsilon_0 p R}{2}}\sum_{j = 1}^\infty \mu(B(\xi_j, \delta))^p  \end{align*} which means that there exists $C_2 > 0$ independent of $N$ where \begin{equation} \|E\|_{S_p} ^p  \leq C_2 e^{- \frac{\epsilon_0 p R}{2}} \sum_{j = 1}^\infty \mu(B(\xi_j, \delta))^p. \label{MainThmForpleq1OffDiagEst5} \end{equation}

Combining $(\ref{MainThmForpleq1DiagEst})$ and $(\ref{MainThmForpleq1OffDiagEst5})$ we have that \begin{equation*}\|T_\mu\|_{S_p}^p  \geq \left(\frac{1}{2} C_1 - C_2 e^{-\frac{\epsilon_0 p R}{2}}\right)  \sum_{j = 1}^\infty \mu(B(\xi_j, \delta))^p \end{equation*} so setting $R$ large enough gives us that \begin{equation}  \sum_{j = 1}^\infty \mu(B(\xi_j, \delta))^p \lesssim \|T_\mu\|_{S_p} ^p \label{MainThmForpleq1FinalEst}\end{equation} for all $\mu$ where $\{\mu(B(\xi_j, \delta))\} \in l^p$.  However, an easy approximation argument gives us that $(\ref{MainThmForpleq1FinalEst})$ holds for all positive Borel measures $\mu$ with $T_\mu \in S_p$.

Finally, since  $(\ref{MainThmForpleq1FinalEst})$ holds for each of the $N$ sublattices of $\{a_m\}$ and since $N$ obviously only depends on $R$, we get that \begin{equation*} \sum_{m = 1}^\infty \mu(B(a_m, \delta))^p \lesssim  \|T_\mu\|_{S_p} ^p \end{equation*} which completes the proof. \end{proof}

\section{The case $p\geq 1$}\label{case2}

In this section we will consider the simpler case of $p\geq 1$.  As with the case $0 < p \leq 1$, the approach is quite similar to the standard Fock space situation.  We will need one preliminary result before we prove our main result.

\begin{lemma} \label{finLpImpTfinSp}
If $p\geq 1$ and $f \in L^p(\C, dv)$, then $T_f \in S_p$.
\end{lemma}

\begin{proof}
Clearly without loss of generality we can assume that $f \geq 0$.  If $\mu = f \, dv$ then clearly we have that $\{\mu(B(a_m, r))\} \in \ell^p$ if $f \in L^p(\C, dv)$ so that $T_f$ is bounded on $\Ftwophi$.  The proof now follows immediately by Lemma \ref{LowerBoundForKernel} in conjunction with the arguments on p. $245$ in \cite{Z} that are used to prove this result in the classical Fock space setting.
\end{proof}

\begin{theorem}
Suppose $\mu\geq0$, $p\geq 1$, and $\mu$ satisfies condition~\eqref{propertyM}. Then the following are equivalent:
\begin{itemize}
\item[(a)] $T_{\mu}$ is in the Schatten class $S_p$;
\item[(b)] $\widetilde{\mu} \in L^p(\C, dv)$;
\item[(c)] $\mu(B(\cdot, r)) \in L^p(\C, dv)$;
\item[(d)] $\{\mu(B(a_m, r))\} \in l^p $.
\end{itemize}
\end{theorem}
\begin{proof}
That (c) and (d) are equivalent  and that both conditions are independent of $r > 0$ was proved in \cite{Z} when $n=1$, though the case $n>1$ is analogous. Note that (a) implies (b) follows from Lemma \ref{BerInLpImpSchat} and the easy proof that (b) implies (d) is analogous to the case $0<p\leq 1$.

We will finish the proof by showing that (c)$\Rightarrow$(a), so suppose that $\hat\mu_r = \mu(B(\cdot, r))\in L^p(\C,dv)$. Then by Fubini's theorem and the reproducing property
\begin{align*}
\langle T_{\hat\mu_r} f,f\rangle_{\Ftwophi} &=
 \int_{\C}  \int_{\C} \chi_{B(w, r)} (z) |f(w)|^2  e^{-2\phi(w)}  \, d\mu(z) \, dv(w)\\
	&= \int_{\C} \left(\int_{B(z,r)} |f(w)|^2 e^{-2\phi(w)} dv(w)\right) d\mu(z)
\end{align*}
 which by Lemma~\ref{PointwiseBound} implies that
\begin{align*}
	\langle T_\mu f,f \rangle_{\Ftwophi}  &= \int_{\C} |f(z)|^2 e^{-2\phi(z)} d\mu(z)\\
	&\leq C_r \int_{\C} \left( \int_{B(z,r)} |f(w) e^{-\phi(w)}|^2 dv(w)\right) d\mu(z)\lesssim  \langle T_{\hat\mu_r}f,f \rangle_{\Ftwophi}
\end{align*}
or $T_\mu \lesssim  T_{\hat\mu_r}$.  The proof is now completed by an application of Lemma \ref{finLpImpTfinSp}.
\end{proof}

 \end{document}